\documentclass[a4paper, 11pt]{elsarticle}
\usepackage{amsmath}
\usepackage{amssymb,esint}
\usepackage{amscd}
\usepackage{xspace}
\usepackage{fancyhdr}
\usepackage{color}
\usepackage{srcltx} 

\newtheorem{theorem}{Theorem}[section]

\newtheorem{definition}[theorem]{Definition}

\newtheorem{lemma}[theorem]{Lemma}

\newtheorem{proposition}[theorem]{Proposition}
\newtheorem{remark}[theorem]{Remark}

\newenvironment{proof}[1][Proof]{\textbf{#1.} }{\hfill\rule{0.5em}{0.5em}}
{\catcode`\@=11\global\let\AddToReset=\@addtoreset
\AddToReset{equation}{section}

\AddToReset{theorem}{section}

\def\nc{\newcommand}

 \def\Om{\Omega}

\nc\pa{\partial}

\nc\CC{\mathbb{C}}
\nc\RR{\mathbb{R}}
\nc\QQ{\mathbb{Q}}
\nc\ZZ{\mathbb{Z}}
\nc\NN{\mathbb{N}}

\begin{document}
\title{Global Lorentz gradient estimates for quasilinear equations with measure data for the strongly singular case: $1<p\leq \frac{3n-2}{2n-1}$}

\author[address1]{Le Cong Nhan}
\ead{nhanlc@hcmute.edu.vn}
\address[address1]{Faculty of Applied Sciences, HCMC University of Technology and Education, Ho Chi Minh City, Vietnam}

\author[myaddress1]{Le Xuan Truong} 
\ead{lxuantruong@gmail.com}
\address[myaddress1]{Department of Mathematics and Statistics, University of Economics Ho Chi Minh City, Vietnam}

\date{}  

\begin{abstract}
In this paper, we study the global regularity estimates in Lorentz spaces for gradients of solutions to quasilinear elliptic equations  with measure data of the form 
	 \begin{eqnarray*}
	\left\{ \begin{array}{rcl}
	-{\rm div}(\mathcal{A}(x, \nabla u))&=& \mu \quad \text{in} ~\Omega, \\
	u&=&0  \quad \text{on}~ \partial \Omega,
	\end{array}\right.
	\end{eqnarray*}
where $\mu$ is a finite signed Radon measure in $\Omega$, $\Om\subset\RR^n$ is a bounded domain such that its complement $\mathbb{R}^n\backslash\Omega$ is uniformly $p$-thick and $\mathcal{A}$ is a Carath\'eodory vector valued function satisfying growth and monotonicity conditions for the strongly singular case $1<p\leq \frac{3n-2}{2n-1}$. Our result extends the earlier results \cite{55Ph0,Tran19} to the strongly singular case $1<p\leq \frac{3n-2}{2n-1}$ and a recent result \cite{HP} by considering rough conditions on the domain $\Omega$ and the nonlinearity $\mathcal{A}$. 

\medskip 

\medskip

\medskip


\medskip

\end{abstract}

\maketitle
\tableofcontents
									
 \section{Introduction and main results} 
 
In this paper we study the gradient regularity of solutions to the following quasilinear elliptic equations with measure data 
 \begin{eqnarray}\label{eq1.1}
 \left\{ \begin{array}{rcl}
 -{\rm div}(\mathcal{A}(x, \nabla u))&=& \mu \quad \text{in} ~\Omega, \\
 u&=&0  \quad \text{on}~ \partial \Omega,
 \end{array}\right.
 \end{eqnarray}
where $\Omega$ is a bounded open subset of $\mathbb{R}^{n}$, ($n\geq2$), and $\mu$ is a finite signed Radon measure in $\Omega$. The nonlinearity $\mathcal{A}:\mathbb{R}^n\times\mathbb{R}^n\to \mathbb{R}^n$ is a Carath\'eodory vector valued function and satisfies the following growth and monotonicity conditions:
	\begin{gather}
		\left|\mathcal{A}\left(x,\xi\right)\right| \leq \beta\left|\xi\right|^{p-1},\label{condi1}\\
		\langle\mathcal{A}\left(x,\xi\right) -\mathcal{A}\left(x,\eta\right),\xi-\eta\rangle\geq \alpha\left(\left|\xi\right|^2+\left|\eta\right|^2\right)^{(p-2)/2}\left|\xi-\eta\right|^2\label{condi2}		
	\end{gather}
for every $\left(\xi,\eta\right)\in \mathbb{R}^n\times\mathbb{R}^n\backslash\{\left(0,0\right)\}$ and a.e. $x\in \mathbb{R}^n$. Here $\alpha$ and $\beta$ are positive constant, and $p$ will be considered in the range
	\begin{align}\label{condi3}
		1<p\leq \frac{3n-2}{2n-1}.
	\end{align}
As the regularity of boundary of $\Omega$ is concerned, we assume a capacity density condition on $\Omega$ which is known weaker than the Reifenberg flatness condition. More precisely, by {\it a capacity density condition on $\Omega$} we mean the complement $\mathbb{R}^n\backslash\Omega$ is {\it uniformly $p$-thick}, that is, there exist constants $c_0,r_0>0$ such that for all $0<t\leq r_0$ and all $x\in \mathbb{R}^n\backslash\Omega$ there holds
	\begin{align}\label{condi4}
		\mathrm{cap}_p\left(\overline{B_t(x)}\cap \left(\mathbb{R}^n\backslash\Omega\right),B_{2t}\left(x\right)\right) \geq c_0\, \mathrm{cap}_p\left(\overline{B_t(x)},B_{2t}\left(x\right)\right).
	\end{align}
Here for a compact set $K\subset B_{2t}\left(x\right)$ we define the $p$-capacity of $K$, $\mathrm{cap}\left(K,B_{2t}(x)\right)$  by
	\begin{align*}
		\mathrm{cap}_p\left(K,B_{2t}(x)\right) =\inf\left\{\int_{\Omega}\left|\nabla \varphi\right|^pdx: \varphi\in C_0^\infty\left(B_{2t}(x)\right)\,\,\text{and}\,\,\varphi\geq \chi_{K}\right\},
	\end{align*}
where $\chi_{K}$ is the characteristic function of $K$.	It is noticed that the domain satisfying \eqref{condi4} includes Lipschitz domains or domain satisfying a uniform exterior corkscrew condition which means that there exist constants $c_0, r_0>0$ such that for all $0<t\leq r_0$ and
all $x\in\mathbb{R}^n\backslash\Omega$, there is $y\in B_t(x)$ such that $B_{t/c_0}(y)\subset \mathbb{R}^n\backslash\Omega$.

\medskip
Under these conditions, our main goal in this paper is to establish the following global gradient estimate in Lorentz spaces
	\begin{align}\label{maingoal}
		\left\|\nabla u \right\|_{L^{s,t}\left(\Omega\right)} \leq C\left\|\mathcal{M}_1 \left(\left|\mu\right|\right)^{1/(p-1)}\right\|_{L^{s,t}\left(\Omega\right)}
	\end{align}
where $s$ lies below or near the natural exponent $p$, i.e., $s<p+\varepsilon$ for some small $\varepsilon$ depending on $n,p,\alpha,\beta$, and $\Omega$, and $t\in\left(0,\infty\right]$. Here $\mathcal{M}_1$ is the fractional maximal function defined for each nonnegative locally
finite measure $\mu$ in $\mathbb{R}^n$ by
	\begin{align*}
		\mathcal{M}_1\left(\mu\right)(x) = \sup\limits_{\rho>0}\frac{\mu\left(B_\rho(x)\right)}{\rho^{n-1}},\quad x\in\mathbb{R}^n.
	\end{align*}
And the Lorentz spaces $L^{s,t}\left(\Omega\right)$, with $1<s<\infty$, and $0<t\leq \infty$, is the set of measurable functions $f$ on $\Omega$ such that
	\begin{align*}
		\left\|f\right\|_{L^{s,t}\left(\Omega\right)} = \left[s\int_{0}^{\infty}\left(\lambda^s\left|\left\{x\in\Omega: \left|f(x)\right|>\lambda\right\}\right|\right)^{{t/s}}\frac{d\lambda}{\lambda}\right]^{1/t}<\infty
	\end{align*}
if $t\neq\infty$. It is also noticed that if $s=t$ then the Lorentz space $L^{s,s}\left(\Omega\right)$ is the usual Lebesgue space $L^s\left(\Omega\right)$. 

\medskip
If $t=\infty$ the space $L^{s,\infty}\left(\Omega\right)$ is the weak $L^s$ or Marcinkiewicz space with quasinorm
	\begin{align*}
		\left\|f\right\| = \sup\limits_{\lambda>0}\lambda\left|\left\{x\in\Omega:\left|f(x)\right|>\lambda\right\}\right|^{1/s}.
	\end{align*}
For $1<r<s<\infty$ then one has
	\begin{align*}
		L^s\left(\Omega\right)\subset L^{s,\infty}\left(\Omega\right)\subset L^r\left(\Omega\right).
	\end{align*}
It is worth mentioning that the local version of \eqref{maingoal} was first obtained by G. Mingione in \cite{55Mi0} for the regular case $2\leq p\leq n$ and then extended several authors in recent. For example, Nguyen Cong Phuc \cite{55Ph0} obtained the global Lorentz estimate for solutions of \eqref{eq1.1} in the `possibly singular' case $2-{1}/{n}<p\leq n$ by using a capacitary density condition on $\Omega$ and the assumptions \eqref{condi1}-\eqref{condi2}. Therein, the author proved the $L^{s,t}(\Omega)$ estimates of solution for all $0<s<p+\varepsilon, 0 < t \leq \infty$ for some $\varepsilon>0$. Afterward a similar result is obtained by M. P. Tran \cite{Tran19} to the singular case $\frac{3n-2}{2n-1}<p\leq 2-\frac{1}{n}$ with the rough conditions on the domain $\Omega$ and the nonlinearity $\mathcal{A}$ in which the author exploits some comparison estimates in \cite{QH4}. In \cite{QH4} by using the good-$\lambda$ type inequality Q-H. Nguyen and Nguyen Cong Phuc proved a global gradient estimates in the weighted Lorentz space for solution to \eqref{eq1.1} in the singular case $\frac{3n-2}{2n-1}<p\leq 2-\frac{1}{n}$ when the nonlinearity $\mathcal{A}$ satisfies the small BMO condition in the $x$-variable and the domain $\Omega$ satisfies the so-called Reifenberg flatness condition. More precisely the authors showed the $L^{s,t}(\Omega)$ estimates of solution for all $0<s<\infty, 0 < t \leq  \infty$. And in a very recent result \cite{HP} under similar conditions on $\mathcal{A}$ and the regularity of $\Omega$, the authors proved a weighted Calder\'on-Zygmund type inequality for $1<p\leq \frac{3n-2}{2n-1}$.

\medskip
Our aim in this paper is to extend these results. By considering the remaining `strongly singular' case $1<p\leq \frac{3n-2}{2n-1}$ and without the hypothesis of Reifenberg flat domain on $\Omega$ and small BMO semi-norms of $\mathcal{A}$, we show that the estimate \eqref{maingoal} holds for all $2-p<s<p+\varepsilon$ and $0<t\leq \infty$ (see Theorem \ref{theo_1.2}). 
        
\medskip
To state our main result, we need some preliminary results on $p$-capacity, a decomposition of measure $\mu$ and the definition of renormalized solution which is can be found in \cite{11DMOP}.

\medskip
For $\mu\in\mathfrak{M}_b(\Omega)$ (the set of finite signed measures in $\Omega$), we will tacitly extend it by zero to $\Omega^c:=\mathbb{R}^n\setminus\Omega$. We let $\mu^+$, $\mu^-$, and $\left|\mu\right|$ be the positive part, negative part, and the total variation of a measure $\mu\in\mathfrak{M}_b(\Omega)$ respectively. Let us also recall that a sequence  $\{\mu_{k}\} \subset
\mathfrak{M}_{b}(\Omega)$ converges to $\mu \in
\mathfrak{M}_{b}(\Omega)$ in the narrow topology of measures if
$$
\lim_{k\rightarrow\infty}\int_{\Omega}\varphi \, d\mu_{k}=\int_{\Omega}\varphi \,d\mu,
$$
for every bounded and continuous function $\varphi$ on $\Omega$. 

\medskip
We denote by $\mathfrak{M}_0(\Omega)$ the set of all measures $\mu\in \mathfrak{M}_b\left(\Omega\right)$ which are absolutely continuous with respect to the $p$-capacity, i.e. which satisfy $\mu\left(B\right)=0$ for every Borel set $B\subset\subset \Omega$ such that $\mathrm{cap}_p\left(B,\Omega\right)=0$. We also denote by $\mathfrak{M}_s(\Omega)$ the set of all measures $\mu\in \mathfrak{M}\left(\Omega\right)$ which are singular with respect to the $p$-capacity. 

\medskip
It is known that any $\mu\in\mathfrak{M}_b(\Omega)$ can be written  uniquely  in the form $\mu=\mu_0+\mu_s$ where $\mu_0\in \mathfrak{M}_0(\Omega)$ and $\mu_s\in \mathfrak{M}_s(\Omega)$ (see \cite{11DMOP}). It is also known  that any  $\mu_0\in \mathfrak{M}_0(\Omega)$ can be written in the form $\mu_0=f-{\rm div}( F)$ where $f\in L^1(\Omega)$ and $F\in L^{\frac{p}{p-1}}(\Omega,\mathbb{R}^n)$.
 
\medskip
To define renormalized solutions, we need some following tools. For $k>0$, we define the usual  two-sided truncation operator $T_k$ by
	$$
	T_k(s)=\max\{\min\{s,k\},-k\}, \qquad s\in\mathbb{R}.
	$$ 
	
For our purpose, the following notion of gradient is needed. If $u$ is a measurable function defined  in $\Omega$, finite a.e., such that $T_k(u)\in W^{1,p}_{loc}(\Omega)$ for any $k>0$, then there exists a measurable function $v:\Omega\to \mathbb{R}^n$ such that $\nabla T_k(u)= v \chi_{\{|u|< k\}}$ a.e. in $\Omega$  for all $k>0$ (see \cite[Lemma 2.1]{bebo}). In this case, we define the gradient $\nabla u$ of $u$ by $\nabla u:=v$. It is known that  $v\in L^1_{loc}(\Omega, \mathbb{R}^n)$ if and only if  $u\in W^{1,1}_{loc}(\Omega)$ and then $v$ is the usual weak gradient of $u$. On the other hand, for $1<p\leq 2-\frac{1}{n}$, by looking at the fundamental solution we see that in general distributional solutions of \eqref{eq1.1} may not even belong to $u\in W^{1,1}_{loc}(\Omega)$.

\medskip  
We now define the renormalized solutions to \eqref{eq1.1} where the right-hand side is assumed to be in $L^1(\Omega)$ or in $\mathfrak{M}_{0}(\Omega)$ (see \cite{11DMOP} for the definition and some other equivalent definitions of renormalized solutions).

\begin{definition} \label{derenormalized} 
Let $\mu=\mu_0+\mu_s\in\mathfrak{M}_b(\Omega)$, with $\mu_0\in \mathfrak{M}_0(\Omega)$ and $\mu_s\in \mathfrak{M}_s(\Omega)$. A measurable  function $u$ defined in $\Omega$ and finite a.e. is called a renormalized solution of \eqref{eq1.1} if $T_k(u)\in W^{1,p}_0(\Omega)$ for any $k>0$, $|{\nabla u}|^{p-1}\in L^r(\Omega)$ for any $0<r<\frac{n}{n-1}$, and $u$ has the following additional property. For any $k>0$ there exist  nonnegative Radon measures $\lambda_k^+, \lambda_k^- \in\mathfrak{M}_0(\Omega)$ concentrated on the sets $\{u=k\}$ and $\{u=-k\}$, respectively, such that $\mu_k^+\rightarrow\mu_s^+$, $\mu_k^-\rightarrow\mu_s^-$ in the narrow topology of measures and  that
  	\[
  	\int_{\{|u|<k\}}\langle A(x,\nabla u),\nabla \varphi\rangle
  	dx=\int_{\{|u|<k\}}{\varphi d}{\mu_{0}}+\int_{\Omega}\varphi d\lambda_{k}%
  	^{+}-\int_{\Omega}\varphi d\lambda_{k}^{-},
  	\]
for every $\varphi\in W^{1,p}_0(\Omega)\cap L^{\infty}(\Omega)$.
\end{definition}

\begin{remark}
	It is known that if $\mu\in \mathfrak{M}_0(\Omega)$ then there is one and only one renormalized solution of \eqref{eq1.1} (see \cite{BGO,11DMOP}). However, to the best of our knowledge, for a general $\mu\in \mathfrak{M}_b(\Omega)$ the uniqueness of renormalized solutions of \eqref{eq1.1} is still an open problem.	
\end{remark}

\begin{remark}
	By \cite[Lemma 4.1]{11DMOP} we have
		\begin{align*}
			\left\|\nabla u\right\|_{L^{\frac{(p-1)n}{n-1},\infty}\left(\Omega\right)} \leq C\left[\left|\mu\right|\left(\Omega\right)\right]^\frac{1}{p-1}
		\end{align*}
	which implies that 
		\begin{align*}
			\left(\frac{1}{R^n}\int_{\Omega}\left|\nabla u\right|^{\gamma_1}\right)^{1/\gamma_1} \leq C_{\gamma_1}\left[\frac{\left|\mu\right|\left(\Omega\right)}{R^{n-1}}\right]^{1/(p-1)}
		\end{align*}	
	for any $0<\gamma_1<\frac{(p-1)n}{n-1}$, where $R=\mathrm{diam}\left(\Omega\right)$.	
\end{remark}

Let us also recall the Hardy-Littlewood maximal function $\mathcal{M}$ is defined for each locally integrable function $f$ in $\mathbb{R}^{n}$ by
\begin{equation*}
\mathcal{M}(f)(x)=\sup_{\rho>0}\fint_{B_\rho(x)}|f(y)|dy, \quad\forall x\in\mathbb{R}^{n}.
\end{equation*}

\begin{remark}
	In \cite{55Gra} the operator $\mathcal{M}$ is bounded from $L^s\left(\mathbb{R}\right)$ to $L^{s,\infty}\left(\mathbb{R}\right)$ for $s\geq 1$, that is,
	\begin{align*}
	\left|\left\{x\in\mathbb{R}^n:\mathcal{M}(f)>\lambda\right\}\right| \leq \frac{C}{\lambda^s}\int_{\mathbb{R}^n}\left|f\right|^sdx\quad\text{for all}\,\,\lambda>0.
	\end{align*}
\end{remark}

\begin{remark}
	In \cite{AH,55Gra} it allows us to present a boundedness property of maximal function $\mathcal{M}$ in the Lorentz space $L^{s,t}\left(\mathbb{R}^n\right)$ for $s > 1$ as follows:
	\begin{align*}
	\left\|\mathcal{M}(f)\right\|_{L^{s,t}\left(\mathbb{R}^n\right)} \leq C \left\|f\right\|_{L^{s,t}\left(\mathbb{R}^n\right)}.
	\end{align*}	
\end{remark}

We are now ready to state the main result of the paper.
\begin{theorem} \label{theo_1.2} 
 	Let $\mu\in \mathfrak{M}_b(\Omega)$ and $1<p\leq \frac{3n-2}{2n-1}$. There exists $\epsilon>0$ such that    for any $2-p<s<p+\epsilon$ and $t\in\left(0,\infty\right]$, then there exists a renormalized solution $u$ to \eqref{eq1.1} such that                           
     \begin{equation}\label{mainbound4}
         \|\nabla u\|_{L^{s,t}(\Omega)}\leq C \|[\mathcal{M}_1(\left|\mu\right|)]^{{1}/{p-1}}\|_{L^{s,t}(\Omega)}.
     \end{equation} 
    Here the constant $C$ depends only  on $n,p,\Lambda,q$, and $diam(\Omega)/r_0$.               
\end{theorem}

For the proofs of the above theorem: we follow the approach developed by \cite{55Mi0,55Ph0} and use some new comparison estimates obtained recently by \cite{HP}, but technically our present is somewhat different form that of \cite{55Mi0,HP}. It is also possible to apply some results developed for quasilinear equations with given measure data, or linear/nonlinear potential and Calder\'on-Zygmund theories (see \cite{bebo, BW1, 11DMOP, 55DuzaMing, Duzamin2, 55MePh2, Mi2,55Mi0, 55QH2, 55Ph0, 55Ph2, 55Ph2-2}), to some new comparison estimates in the singular case $1<p\leq \frac{3n-2}{2n-1}$.	


\medskip
The paper is organized as follows. In Section \ref{sec-2} we present some important comparison estimates that are needed for the proof of main result and its applications are given in Sections \ref{sec-3}. In Section \ref{sec-4} we complete the proof of Theorem \ref{theo_1.2}

\section{Local interior and boundary comparison estimates}\label{sec-2}
Let $u\in W_{loc}^{1,p}(\Omega)$ be a solution of \eqref{eq1.1} and for each ball $B_{2R}=B_{2R}(x_0)\subset\subset\Omega$, we consider the unique solution $w\in u+W_{0}^{1,p}\left(B_{2R}\right)$ to the  equation 
    \begin{equation}
    \label{eq2.1}\left\{ \begin{array}{rcl}
    - \operatorname{div}\left( {\mathcal{A}(x,\nabla w)} \right) &=& 0 \quad \text{in} \quad B_{2R}, \\ 
    w &=& u\quad \text{on} \quad \partial B_{2R}.  
    \end{array} \right.
    \end{equation}
Then the following well-known version of Gehring's lemma holds for function $w$ defined above, see \cite[Theorem 6.7 and Remark 6.12]{Gi} and also \cite[Lemma 2.1]{55Ph0}.
	
	\begin{lemma}\label{lem_Gehring}
		Let $w$ be the solution to \eqref{eq2.1}. Then there exists a constant $\theta_0=\theta_0\left(n,p,\alpha,\beta\right)>1$ such that for any $t\in \left(0,p\right]$ and any balls $B_\rho(y)\subset B_{2R}(x_0)$ the following reverse H\"older type inequality holds 
			\begin{align*}
				\left(\fint_{B_{\rho/2}(y)}\left|\nabla w\right|^{p\theta_0}dx\right)^{1/p\theta_0} \leq C\left(\fint_{B_{\rho}(y)}\left|\nabla w\right|^{t}dx\right)^{t}
			\end{align*}
		with a constant $C=C\left(n,p,\alpha,\beta,t\right)$.	
	\end{lemma}	
	
The next comparison estimate gives an estimate for the difference $\nabla u-\nabla w$ in terms of the total variation of $\mu$ in $B_{2R}$ and the norm of $\nabla u$ in $L^{2-p}(B_{2R})$ in the "strongly singular" case $1<p\leq \frac{3n-2}{2n-1}$. This result was proved by Q-H. Nguyen \cite[Lemma 2.1]{HP}. Similar estimates for the other case, i.e., $p>\frac{3n-2}{2n-1}$ was given in \cite{Mi2,55DuzaMing,Duzamin2,QH4}. 
		
    \begin{lemma}\label{lem_2.2}
    	Let $u$ and $w$ be solution to \eqref{eq1.1} and \eqref{eq2.1} respectively and assume that   $1<p\leq \frac{3n-2}{2n-1}$. Then
    	\begin{align*}
    	\left(\fint_{B_{2R}}|\nabla (u-w)|^{\gamma_1}dx\right)^{{1}/{\gamma_1}} \leq C\left(\frac{|\mu|(B_{2R})}{R^{n-1}}\right)^{{1}/{(p-1)}}+\frac{|\mu|(B_{2R})}{R^{n-1}}\fint_{B_{2R}}|\nabla u|^{2-p}dx,
    	\end{align*}
    	for any $0<\gamma_1<\frac{n(p-1)}{n-1}$.
    \end{lemma}

\medskip
Lemmas \ref{lem_Gehring} and \ref{lem_2.2} can be extended  up to  boundary. As $\mathbb{R}^n\backslash\Omega$ is uniformly $p$-thick with constants $c_0, r_0>0$, there exists $1<p_0=
p_0(n, p, c_0)<p$ such that $\mathbb{R}^n\backslash\Omega$ is uniformly $p_0$-thick with constants $c_*=c(n, p, c_0)$
and $r_0$. This is by now a classical result due to Lewis \cite{Le88} (see also \cite{Mik}). Moreover, $p_0$ can be chosen near $p$ so that $p_0\in (np/(n+p), p)$. Thus, since $p_0<n$, we have
	\begin{align*}
		\mathrm{cap}_{p_0}\left(\overline{B_t(x)}\cap \left(\mathbb{R}^n\backslash\Omega\right),B_{2t}(x)\right) \geq c_*\mathrm{cap}_{p_0}\left(\overline{B_t(x)},B_{2t}(x)\right)
		\geq C\left(n,p,c_0\right)t^{n-p_0},
	\end{align*}
for all $0<t\leq r_0$ and for all $x\in\mathbb{R}^n\backslash\Omega$.

\medskip
Fix $x_0\in \partial\Omega$ and $0<R\leq r_0/10$. With $u\in W_0^{1,p}(\Omega)$ being a solution to \eqref{eq1.1}, we now consider the unique solution $w\in u+W_{0}^{1,p}\left(\Omega_{10R}(x_0)\right)$ to the following equation 
      \begin{equation}
      \label{eq2.3}\left\{ \begin{array}{rcl}
      - \operatorname{div}\left( {\mathcal{A}(x,\nabla w)} \right) &=& 0 \quad ~~~\text{in}\quad \Omega_{10R}(x_0), \\ 
      w &=& u\quad \quad \text{on} \quad \partial \Omega_{10R}(x_0),
      \end{array} \right.
      \end{equation}
where we define $\Omega_{10R}(x_0)=\Omega\cap B_{10R}(x_0)$ and extend $u$ by zero to $\mathbb{R}^n\backslash\Omega$ and $w$ by $u$ to $\mathbb{R}^n\backslash\Omega_{10R}(x_0)$. 

\medskip
Then we have the following version of Lemma \ref{lem_Gehring} up to the boundary, see \cite[Lemma 2.5]{55Ph0} for its proof.

\begin{lemma}\label{lem_2.3}
	Let $w$ be solution to \eqref{eq2.3}. Then there exists a constant $\theta_0=\theta_0\left(n,p,\alpha,\beta\right)>1$ such that for any $t\in \left(0,p\right]$ the following reverse H\"older type inequality 
	\begin{align*}
	\left(\fint_{B_{\rho/2}(y)}\left|\nabla w\right|^{p\theta_0}dx\right)^{1/p\theta_0} \leq C\left(\fint_{B_{3\rho}(y)}\left|\nabla w\right|^{t}dx\right)^{t}
	\end{align*}
	holds for any balls $B_{3\rho}(y)\subset B_{10R}(x_0)$  with a constant $C=C\left(n,p,\alpha,\beta,t\right)$.	
\end{lemma}

As a consequence, we have another version of reverse H\"older type inequality. 

  \begin{lemma}
	Let $w$ be solution to \eqref{eq2.3}. Then there exists a constant $\theta_0=\theta_0\left(n,p,\alpha,\beta\right)>1$ such that for any $t\in \left(0,p\right]$ and for $0<\sigma_1<\sigma_2<1$ it holds that
	\begin{align*}
	\left(\fint_{B_{\sigma_1}\rho(y)}\left|\nabla w\right|^{p\theta_0}dx\right)^{1/p\theta_0} \leq C\left(\fint_{B_{\sigma_2\rho}(y)}\left|\nabla w\right|^{t}dx\right)^{t}
	\end{align*}
	for any balls $B_{\rho}(y)\subset B_{10R}(x_0)$  with a constant $C=C\left(n,p,\alpha,\beta,t,\rho_1,\rho_2\right)$.	
\end{lemma}

\begin{proof}
	Let $x_1,x_2,...,x_m\in B_{\sigma_1}(0)$ be such that	
		\begin{align*}
			B_{\sigma_1}(0) \subset \bigcup\limits_{i=1}^m B_{\frac{\sigma_2-\sigma_1}{100}}(x_i)
		\end{align*}
	For $\rho>0$, let $B_{\rho}(y)\subset B_{10R}(x_0)$, then we find that
		\begin{align}\label{eq2.4}
		B_{\sigma_1\rho}(y) \subset \bigcup\limits_{i=1}^m B_{\frac{\left(\sigma_2-\sigma_1\right)\rho}{100}}(y+\rho x_i)
		\end{align}
	Also noticed that since $B_{\frac{6\left(\sigma_2-\sigma_1\right)\rho}{100}}(y+\rho x_i) \subset B_{\sigma_2\rho}(y)\subset B_{\rho}(y)$ for all $i=1,...,m$, by Lemma \ref{lem_2.3}, there exist $\theta_0=\theta_0\left(n,p,\alpha,\beta\right)$ and $C=C\left(n,p,\alpha,\beta,t\right)$ such that for each $t\in\left(0,p\right]$ it holds that, for all $i=1,..,m$
		\begin{align}\label{eq2.5}
			\left(\fint_{B_{\frac{\left(\sigma_2-\sigma_1\right)\rho}{100}}(y+\rho x_i)}\left|\nabla w\right|^{p\theta_0}dx\right)^{1/p\theta_0} \leq C\left(\fint_{B_{\frac{6\left(\sigma_2-\sigma_1\right)\rho}{100}}(y+\rho x_i)}\left|\nabla w\right|^{t}dx\right)^{t}.
		\end{align}
	From \eqref{eq2.4} and \eqref{eq2.5} for any $t\in\left(0,p\right]$ we have
		\begin{align*}
			\left(\fint_{B_{\sigma_1}\rho(y)}\left|\nabla w\right|^{p\theta_0}dx\right)^{1/p\theta_0} \leq& C\sum_{i=1}^{m}\left(\fint_{B_{\frac{\left(\sigma_2-\sigma_1\right)\rho}{100}}(y+\rho x_i)}\left|\nabla w\right|^{p\theta_0}dx\right)^{1/p\theta_0} \\
			\leq & C\sum_{i=1}^{m}\left(\fint_{B_{\frac{6\left(\sigma_2-\sigma_1\right)\rho}{100}}(y+\rho x_i)}\left|\nabla w\right|^{t}dx\right)^{t}\\
			\leq &C\left(\fint_{B_{{\sigma_2\rho}}(y)}\left|\nabla w\right|^{t}dx\right)^{t}.
		\end{align*}
	Thus the proof is complete.	
\end{proof}

\medskip
We also present here the counterpart of Lemma \ref{lem_2.2} up to the boundary, see \cite[Lemma 2.3]{HP}.
  
  \begin{lemma}\label{lem_2.4} Let $1<p\leq \frac{3n-2}{2n-1}$, and  let $u, w$ be solution to \eqref{eq1.1} and \eqref{eq2.3} respectively. Then we have
  	\begin{align*}\nonumber
  \left(	\fint_{B_{10R}(x_0)}|\nabla (u-w)|^{\gamma_1}dx\right)^{{1}/{\gamma_1}}&\leq C\left[\frac{|\mu|(B_{10R}(x_0))}{R^{n-1}}\right]^{{1}/{(p-1)}}\\
&\qquad +	C \frac{|\mu|(B_{10R}(x_0))}{R^{n-1}}\fint_{B_{10R}(x_0)}|\nabla u|^{2-p}dx,
  	\end{align*}
for any $0<\gamma_1<\frac{(p-1)n}{n-1}$.
  \end{lemma}


\section{Applications of comparison estimates}\label{sec-3}
Our approach to Theorem \ref{theo_1.2} is based on the following technical lemma which allows one to work with balls instead of cubes. A version of this lemma appeared for the first time in \cite{Wa}. It can be viewed as a version of the Calder\'on-Zygmund-
Krylov-Safonov decomposition that has been used in \cite{CaPe} and \cite{55Mi0}. A proof of this lemma, which uses Lebesgue’s differentiation theorem and the standard Vitali covering lemma, can be found in \cite{BW1} with obvious modifications to fit the setting here.

\begin{lemma}\label{lem_6.1}
	Assume that $A\subset\mathbb{R}^n$ is a measurable set for which there exist $c_1,r_1>0$ such that
		\begin{align}\label{lem6.1_condi}
			\left|B_t(x)\cap A\right|\geq c_1\left|B_t(x)\right|
		\end{align}
	holds for all $x\in A$ and $0<t\leq r_1$. Fix $0<r\leq r_1$ and let $C\subset D\subset A$ be measurable sets for which there exists $0<\varepsilon<1$ such that
		\begin{itemize}
			\item[(i)] $\left|C\right|<\varepsilon r^n\left|B_1\right|$;
			\item[(ii)] for all $x\in A$ and $\rho\in\left(0,r\right]$, if $\left|C\cap B_\rho(x)\right|\geq \varepsilon\left|B_\rho(x)\right|$, then $B_\rho(x)\cap A\subset D$.
		\end{itemize}
	Then we have the estimate $\left|C\right|\leq \frac{\varepsilon}{c_1}\left|D\right|$.
\end{lemma}

In order to apply Lemma \ref{lem_6.1} we need the following proposition, whose proof relies essentially on the comparison estimates in the previous section.

\begin{proposition}\label{prop_6.2}
	There exist constants $A$, $\theta_0>1$, depending only on $n, p, \alpha, \beta$, and $c_0$, so that the following holds for any $T >1$ and $\lambda>0$. Let $u$ be a solution of \eqref{eq1.1} with
	$\mathcal{A}$ satisfying \eqref{condi1} and \eqref{condi2}. Assume that for some ball $B_\rho\left(y\right)$ with $10\rho\leq r_0$ we have	
		\begin{align}\label{eq6.1}
			&\left\{x\in B_\rho(y):\mathcal{M}\left(\chi_\Omega\left|\nabla u\right|^{\gamma_1}\right)(x)^{1/\gamma_1}\leq \lambda,\,\,\mathcal{M}_1\left(\chi_\Omega\left|\mu\right|\right)(x)^{1/(p-1)}\leq T^{-\gamma}\lambda \right\}\neq \emptyset,
		\end{align}
	where $0<\gamma_1<\min\left\{\frac{p\theta_0}{p-1},\frac{(p-1)n}{n-1}\right\}$ and $\gamma=\frac{p\theta_0}{\gamma_1\left(p-1\right)}-1>0$.
	Then
		\begin{align}\label{eq6.2}
			&\left|\left\{x\in B_\rho(y):\mathcal{M}\left(\chi_\Omega\left|\nabla u\right|^{\gamma_1}\right)(x)^{1/\gamma_1}> AT\lambda,\right.\right.\notag\\
			&\qquad\qquad\quad~~\left.\left.\mathcal{M}\left(\chi_\Omega\left|\nabla u\right|^{2-p}\right)(x)^{1/(2-p)}\leq T\lambda \right\}\right|\leq  T^{-p\theta_0}\left|B_\rho(y)\right|
		\end{align}	
\end{proposition}
\begin{proof} From the assumption \eqref{eq6.1}, we imply that there exists $x_1\in B_\rho(y)$ such that
		\begin{align}\label{eq6.3}
			\left[\mathcal{M}\left(\chi_\Omega\left|\nabla u\right|^{\gamma_1}\right)(x_1)\right]^{1/\gamma_1}\leq \lambda \quad\text{and}\quad \left[\mathcal{M}\left(\chi_\Omega\left|\mu\right|^{\gamma_1}\right)(x_1)\right]^{1/(p-1)}\leq T^{-\gamma}\lambda.
		\end{align}
	On the other hand \eqref{eq6.2}	is trivial if the set on the left hand side is empty, so we may assume that there is $x_2\in B_\rho(y)$ so that
		\begin{align}\label{eq6.3_1}
			\left[\mathcal{M}\left(\chi_\Omega\left|\nabla u\right|^{2-p}\right)(x_2)\right]^{1/(2-p)}\leq T\lambda.
		\end{align}		
	 It follows from \eqref{eq6.3} that for any $x\in B_\rho(y)$
		\begin{align}\label{eq6.4}
			\left[\mathcal{M}(|\chi_\Omega\nabla u|^{\gamma_1})(x)\right]^{{1}/{\gamma_1}}\leq \max\{\left[\mathcal{M}\left(\chi_{B_{2\rho}(y)\cap\Omega}|\nabla u|^{\gamma_1}\right)(x)\right]^{{1}/{\gamma_1}},3^{n}\lambda\}.
		\end{align}
	From the last estimate, we observe that \eqref{eq6.2} is also trivial provided $A\geq 3^n$ and $B_{4\rho}(y)\subset \mathbb{R}^n\backslash\Omega$. Thus it suffices to consider two cases: the case $B_{4\rho}(y)\subset\Omega$ and the case $B_{4\rho}(y)\cap \partial\Omega\neq \emptyset$.
	
	\medskip
	{\bf 1. The case $B_{4\rho}(y)\subset\Omega$.} Let $w\in u+W_0^{1,p}\left(B_{4\rho}(y)\right)$ be the unique solution to the Dirichlet problem
		\begin{align*}
			\left\{
			\begin{array}{rl}
			\mathrm{div}\mathcal{A}\left(x,\nabla w\right)=0&\text{in}\quad B_{4\rho}(y),\\
			w=u&\text{on}\quad\partial B_{4\rho(y)}.
			\end{array}
			\right.
		\end{align*}
	By Chebyshev inequality and weak type (1,1) estimate for the maximal function we have
		\begin{align*}
			&\left|\left\{x\in B_\rho(y):\mathcal{M}\left(\chi_{B_{2\rho}(y)}\left|\nabla u\right|^{\gamma_1}\right)(x)^{1/\gamma_1}>AT\lambda\right\}\right| \\
			&\qquad\qquad\leq \left|\left\{x\in B_\rho(y):\mathcal{M}\left(\chi_{B_{2\rho}(y)}\left|\nabla w\right|^{\gamma_1}\right)(x)^{1/\gamma_1}>AT\lambda/2\right\}\right|\\
			&\qquad\qquad\quad+\left|\left\{x\in B_\rho(y):\mathcal{M}\left(\chi_{B_{2\rho}(y)}\left|\nabla u-\nabla w\right|^{\gamma_1}\right)(x)^{1/\gamma_1}>AT\lambda/2\right\}\right|\\
			&\qquad\qquad\leq C(AT\lambda)^{-p\theta_0}\int_{B_{2\rho}(y)}\left|\nabla w\right|^{p\theta_0}dx+C(AT\lambda)^{-\gamma_1}\int_{B_{2\rho}(y)}\left|\nabla u -\nabla w\right|^{\gamma_1}dx
		\end{align*}
	Using Lemma \ref{lem_Gehring} we have
		\begin{align*}
			\left(\fint_{B_{2\rho}(y)}\left|\nabla w\right|^{p\theta_0}dx\right)^{\gamma_1/p\theta_0} \leq& C\fint_{B_{4\rho}(y)}\left|\nabla w\right|^{\gamma_1}dx\\
			\leq& C\fint_{B_{4\rho}(y)}\left|\nabla u\right|^{\gamma_1}dx + C\fint_{B_{4\rho}(y)}\left|\nabla u-\nabla w\right|^{\gamma_1}dx
		\end{align*}	
	and hence
		\begin{align}\label{eq6.5}
			&\left|\left\{x\in B_{\rho}(y): \mathcal{M}\left(\chi_{B_{2\rho}(y)}\left|\nabla u\right|^{\gamma_1}\right)(x)^{1/\gamma_1}>AT\lambda\right\}\right| \notag\\
			&\qquad\qquad\leq C\left(AT\lambda\right)^{-p\theta_0}\left|B_{\rho}(y)\right|\left(\fint_{B_{4\rho}(y)}\left|\nabla u\right|^{\gamma_1}dx\right)^{p\theta_0/\gamma_1}\notag\\
			&\qquad\qquad\quad+C\left(AT\lambda\right)^{-p\theta_0}\left|B_{\rho}(y)\right| \left(\fint_{B_{4\rho}(y)}\left|\nabla u-\nabla w\right|^{\gamma_1}dx\right)^{p\theta_0/\gamma_1}\notag\\
			&\qquad\qquad\quad+C\left(AT\lambda\right)^{-\gamma_1}\left|B_{\rho}(y)\right| \fint_{B_{4\rho}(y)}\left|\nabla u-\nabla w\right|^{\gamma_1}dx.
		\end{align}	
	On the other hand, by Lemma \ref{lem_2.2} we have that
		\begin{align*}
			&\left(\fint_{B_{4\rho}(y)}|\nabla (u-w)|^{\gamma_1}dx\right)^{\frac{1}{\gamma_1}}\notag\\
			&\qquad\leq C\left(\frac{|\mu|(B_{4\rho}(y))}{\rho^{n-1}}\right)^{\frac{1}{p-1}}+\frac{|\mu|(B_{4\rho}(y))}{\rho^{n-1}}\fint_{B_{4\rho}(y)}|\nabla u|^{2-p}dx\notag\\
			&\qquad\leq C\left(\frac{|\mu|(B_{5\rho}(x_1))}{\rho^{n-1}}\right)^{\frac{1}{p-1}}+C\frac{|\mu|(B_{5\rho}(x_1))}{\rho^{n-1}}\fint_{B_{5\rho}(x_1)}|\nabla u|^{2-p}dx,
		\end{align*}
	which implies, due to \eqref{eq6.3} and \eqref{eq6.3_1}
		\begin{align}\label{eq6.6}
		&\fint_{B_{4\rho}(y)}|\nabla (u-w)|^{\gamma_1}dx\notag\\
		&\qquad\leq C\left(\frac{|\mu|(B_{5\rho}(x_1))}{\rho^{n-1}}\right)^{\frac{\gamma_1}{p-1}}+C\left(\frac{|\mu|(B_{5\rho}(x_1))}{\rho^{n-1}}\right)^{\gamma_1}\left(\fint_{B_{7\rho}(x_2)}|\nabla u|^{2-p}dx\right)^{\gamma_1}\notag\\
		&\qquad\leq C\left[T^{-\gamma\gamma_1} + T^{-\gamma\gamma_1(p-1)+\gamma_1(2-p)}\right]\lambda^{\gamma_1}\notag\\
		&\qquad\leq CT^{-\gamma\gamma_1(p-1)+\gamma_1(2-p)}\lambda^{\gamma_1}.
		\end{align}
	
	Combining \eqref{eq6.5} and \eqref{eq6.6} and noting that $T>1$, one has
	\begin{align*}
		&\left|\left\{x\in B_\rho(y):\mathcal{M}\left(\chi_{B_{2\rho}(y)}\left|\nabla u\right|^{\gamma_1}\right)(x)^{1/\gamma_1}> AT\lambda,\right.\right.\notag\\
		&\qquad\qquad\qquad\left.\left.\mathcal{M}\left(\chi_\Omega\left|\nabla u\right|^{2-p}\right)(x)^{1/(2-p)}\leq T\lambda \right\}\right|\notag\\
		&\qquad\qquad\leq C\left[\left(AT\right)^{-p\theta_0} + A^{-p\theta_0}T^{-p\theta_0\left(\gamma+1\right)\left(p-1\right)}+ A^{-\gamma_1}T^{-\gamma_1\left(\gamma+1\right)(p-1)}\right] \left|B_\rho(y)\right|\notag\\
		&\qquad\qquad\leq \left[CA^{-p\theta_0} + CA^{-\gamma_1}\right]T^{-p\theta_0}\left|B_\rho(y)\right|.
	\end{align*}
By choosing $A\geq \max\left\{3^n,\left(4C\right)^{1/\gamma_1}\right\}$  we have that
	\begin{align*}
		&\left|\left\{x\in B_\rho(y):\mathcal{M}\left(\chi_{B_{2\rho}(y)}\left|\nabla u\right|^{\gamma_1}\right)(x)^{1/\gamma_1}> AT\lambda,\right.\right.\\
		&\left.\left.\qquad\qquad\qquad\mathcal{M}\left(\chi_\Omega\left|\nabla u\right|^{2-p}\right)(x)^{1/(2-p)}\leq T\lambda \right\}\right|\leq\frac{1}{2}T^{-p\theta_0}\left|B_\rho(y)\right|,
	\end{align*}
which in view of \eqref{eq6.4} yields \eqref{eq6.2}.

\medskip
{\bf 2. The case $B_{4\rho}(y)\cap\partial\Omega\neq \emptyset$.} Let $x_3\in\partial\Omega$ be a boundary point such that $\left|y-x_3\right|=\mathrm{dist}\left(y,\partial\Omega\right)<4\rho$. And let us define $w\in u+ W_0^{1,p}\left(\Omega_{10\rho}(x_3)\right)$ to be the unique solution to the Dirichlet problem 	
	\begin{align*}
	\left\{
	\begin{array}{rl}
	\mathrm{div}\mathcal{A}\left(x,\nabla w\right)=0&\text{in}\quad\Omega_{10\rho}(x_3),\\
	w=u&\text{on}\quad\partial \Omega_{10\rho}(x_3).
	\end{array}
	\right.
	\end{align*}
Here we also extend $u$ by zero to $\mathbb{R}^n\backslash\Omega$	and then extend $w$ by $u$ to $\mathbb{R}\backslash\Omega_{10\rho}(x_3)$. Using similar argument as in \eqref{eq6.5} in which Lemma \ref{lem_2.3} is exploited instead of Lemma \ref{lem_Gehring}, we have
	\begin{align}\label{eq6.8}
	&\left|\left\{x\in B_{\rho}(y): \mathcal{M}\left(\chi_{B_{2\rho}(y)}\left|\nabla u\right|^{\gamma_1}\right)(x)^{1/\gamma_1}>AT\lambda\right\}\right| \notag\\
	&\qquad\qquad\leq C\left(AT\lambda\right)^{-p\theta_0}\left|B_{\rho}(y)\right|\left(\fint_{B_{6\rho}(y)}\left|\nabla u\right|^{\gamma_1}dx\right)^{p\theta_0/\gamma_1}\notag\\
	&\qquad\qquad\quad+C\left(AT\lambda\right)^{-p\theta_0}\left|B_{\rho}(y)\right| \left(\fint_{B_{6\rho}(y)}\left|\nabla u-\nabla w\right|^{\gamma_1}dx\right)^{p\theta_0/\gamma_1}\notag\\
	&\qquad\qquad\quad+C\left(AT\lambda\right)^{-\gamma_1}\left|B_{\rho}(y)\right| \fint_{B_{6\rho}(y)}\left|\nabla u-\nabla w\right|^{\gamma_1}dx.
	\end{align}
On the other hand, since
	\begin{align*}
		B_{6\rho}(y) \subset B_{10\rho}(x_3) \subset B_{14\rho}(y) \subset B_{15\rho}(x_1)\subset B_{16\rho}(x_2)
	\end{align*}
it follows from Lemma \ref{lem_2.4} that
	\begin{align}\label{eq6.9}
		&\fint_{B_{6\rho}(y)}|\nabla (u-w)|^{\gamma_1}dx\leq CT^{-p\theta_0+\gamma_1}\lambda^{\gamma_1}
	\end{align}
	Combining \eqref{eq6.8} and \eqref{eq6.9} and the fact that $T>1$ we have
		\begin{align*}
		&\left|\left\{x\in B_\rho(y):\mathcal{M}\left(\chi_{B_{2\rho}(y)}\left|\nabla u\right|^{\gamma_1}\right)(x)^{1/\gamma_1}> AT\lambda,\right.\right.\\
		&\quad\quad\left.\left.\mathcal{M}\left(\chi_\Omega\left|\nabla u\right|^{2-p}\right)(x)^{1/(2-p)}\leq T\lambda \right\}\right|\leq \left[CA^{-p\theta_0} + CA^{-\gamma_1}\right]T^{-p\theta_0}\left|B_\rho(y)\right|.
		\end{align*}
	By choosing $A\geq \max\left\{3^n,\left(4C\right)^{1/\gamma_1}\right\}$ we arrive at
		\begin{align}\label{eq6.10}
		&\left|\left\{x\in B_\rho(y):\mathcal{M}\left(\chi_{B_{2\rho}(y)}\left|\nabla u\right|^{\gamma_1}\right)(x)^{1/\gamma_1}> AT\lambda,\right.\right.\notag\\
		&\qquad\qquad\qquad\left.\left.\mathcal{M}\left(\chi_\Omega\left|\nabla u\right|^{2-p}\right)(x)^{1/(2-p)}\leq T\lambda \right\}\right|\leq\frac{1}{2}T^{-p\theta_0}\left|B_\rho(y)\right|
		\end{align}
	Thus \eqref{eq6.2} follows from \eqref{eq6.4} and \eqref{eq6.10}.
\end{proof}

\medskip
The Proposition \ref{prop_6.2} can be rewritten as follows.
\begin{proposition}\label{prop_6.3}
	There exist constants $A$, $\theta_0>1$, depending only on $n, p, \alpha, \beta$, and $c_0$, so that the following holds for any $T >1$ and $\lambda>0$. Let $u$ be a solution of \eqref{eq1.1} with
	$\mathcal{A}$ satisfying \eqref{condi1} and \eqref{condi2}. Suppose that for some ball $B_{\rho}(y)$ with $10\rho\leq r_0$ we have
		\begin{align*}
			&\left|\left\{x\in B_\rho(y):\mathcal{M}\left(\chi_{\Omega}\left|\nabla u\right|^{\gamma_1}\right)(x)^{1/\gamma_1}> AT\lambda,\right.\right.\\
			&\qquad\qquad\qquad\left.\left.\mathcal{M}\left(\chi_\Omega\left|\nabla u\right|^{2-p}\right)(x)^{1/(2-p)}\leq T\lambda \right\}\right| \geq T^{-p\theta_0}\left|B_\rho(y)\right|.
		\end{align*}
	Then	
		\begin{align*}
			B_\rho(y)\subset \left\{x\in \mathbb{R}^n:\mathcal{M}\left(\chi_\Omega\left|\nabla u\right|^{\gamma_1}\right)(x)^{1/\gamma_1}>\lambda\,\,\text{or}\,\,\mathcal{M}_1\left(\chi_\Omega\left|\nabla u\right|\right)(x)^{1/(p-1)}> T^{-\gamma}\lambda\right\},
		\end{align*}
	where $\gamma$ and $\gamma_1$ as in Proposition \ref{prop_6.2}.
\end{proposition}

With the aid of Lemma \ref{lem_6.1} and Proposition \ref{prop_6.3} we get the main result of this section which is used later.
\begin{lemma}\label{lem_6.4}
	There exist constants $A$, $\theta_0>1$, depending only on $n, p, \alpha, \beta$, and $c_0$, so that the following holds for any $T >1$. Let $u$ be a solution of \eqref{eq1.1} with
	$\mathcal{A}$ satisfying \eqref{condi1} and \eqref{condi2}. Let $B_0$ be a ball of radius $R_0$. Fix a real number $0<r\leq \min\{r_0, 2R_0\}/10$ and suppose that there exists $\Lambda>0$ such that
		\begin{align}\label{eq6.11}
			\left|\left\{x\in\mathbb{R}^n:\mathcal{M}\left(\chi_\Omega\left|\nabla u\right|^{\gamma_1}\right)(x)^{1/\gamma_1}>\Lambda \right\}\right| < T^{-p\theta_0}r^n\left|B_1\right|.
		\end{align}
	Then for any integer $i\geq 0$ it holds that 
		\begin{align*}
			&\left|\left\{x\in B_0:\mathcal{M}\left(\chi_\Omega\left|\nabla u\right|^{\gamma_1}\right)(x)^{1/\gamma_1}> \Lambda\left(AT\right)^{i+1}\right\}\right| \notag\\
			&\qquad\leq c(n)T^{-p\theta_0}\left|\left\{x\in B_0:\mathcal{M}\left(\chi_\Omega\left|\nabla u\right|^{\gamma_1}\right)(x)^{1/\gamma_1}>\Lambda(AT)^i\right\}\right| \notag\\
			&\qquad\quad+ c(n)\left|\left\{x\in B_0:\mathcal{M}_1\left(\chi_\Omega\left|\mu\right|\right)(x)^{1/(p-1)}>\Lambda T^{-\gamma}(AT)^i\right\}\right|\notag\\
			&\qquad\quad+ \left|\left\{x\in B_0:\mathcal{M}\left(\chi_\Omega\left|\nabla u\right|^{2-p}\right)(x)^{1/(2-p)}>\Lambda T(AT)^{i+1}\right\}\right|.						
		\end{align*}
\end{lemma}
\begin{proof}
	Let $A$ and $\theta_0>1$ be as in Proposition \ref{prop_6.3} and set
		\begin{align*}
			&C=\left\{x\in B_0:\mathcal{M}\left(\chi_\Omega\left|\nabla u\right|^{\gamma_1}\right)(x)^{1/\gamma_1}>\Lambda\left(AT\right)^{i+1} ,\right.\notag\\
			&\qquad\qquad\qquad~~\left.\text{and}\,\,\mathcal{M}\left(\chi_\Omega\left|\nabla u\right|^{2-p}\right)(x)^{1/(2-p)}\leq \Lambda T\left(AT\right)^{i+1} \right\},
		\end{align*}
	and
		\begin{align*}
			&D=\left\{x\in B_0:\mathcal{M}\left(\chi_\Omega\left|\nabla u\right|^{\gamma_1}\right)(x)^{1/\gamma_1}>\Lambda\left(AT\right)^i\right.\notag\\
			&\qquad\qquad\qquad~~\left.\text{or}\,\,\mathcal{M}_1\left(\chi_\Omega\left|\mu\right|\right)(x)^{1/(p-1)}> \Lambda T^{-\gamma}\left(AT\right)^i \right\}.
		\end{align*}	
	It is first noticed that since $AT>1$ we deduce from \eqref{eq6.11} that 
		$$
		\left|C\right|\leq T^{-p\theta_0}r^n\left|B_1\right|.
		$$ 
	On the other hand, if $x\in B_0$ and $\rho\in \left(0,r\right]$ hold $\left|C\cap B_\rho(x)\right|\geq T^{-p\theta_0}\left|B_\rho(x)\right|$, then $10\rho\leq r_0$ and therefore by applying Proposition \ref{prop_6.3} with $\lambda=\Lambda\left(AT\right)^i$ we have
		\begin{align*}
			B_\rho(x)\cap B_0\subset D.
		\end{align*}
	Applying Lemma \ref{lem_6.1} with $A=B_0$ and $\varepsilon=T^{-p\theta_0}$ with noting that the condition \eqref{lem6.1_condi} holds for all $0<t<2R_0$ we obtain
		\begin{align*}
			\left|C\right| \leq& c(n)T^{-p\theta_0}\left|D\right|\\
			\leq &c(n)T^{-p\theta_0}\left|\left\{x\in B_0:\mathcal{M}\left(\chi_\Omega\left|\nabla u\right|^{\gamma_1}\right)(x)^{1/\gamma_1}>\Lambda\left(AT\right)^i\right\}\right|\\
			&+c(n)T^{-p\theta_0}\left|\left\{x\in B_0:\mathcal{M}_1\left(\chi_\Omega\left|\mu\right|\right)(x)^{1/(p-1)}> \Lambda T^{-\gamma}\left(AT\right)^i \right\}\right|.
		\end{align*}
Since $T>1$	the last inequality implies
	\begin{align*}
		&\left|\left\{x\in B_0:\mathcal{M}\left(\chi_\Omega\left|\nabla u\right|^{\gamma_1}\right)(x)^{1/\gamma_1}> \Lambda\left(AT\right)^{i+1}\right\}\right| \\
		&\qquad\qquad\leq \left|C\right| + \left|\left\{x\in B_0:\mathcal{M}\left(\chi_\Omega\left|\nabla u\right|^{2-p}\right)(x)^{1/(2-p)}> \Lambda T\left(AT\right)^{i+1}\right\}\right|\\
		&\qquad\qquad\leq c(n)T^{-p\theta_0}\left|\left\{x\in B_0:\mathcal{M}\left(\chi_\Omega\left|\nabla u\right|^{\gamma_1}\right)(x)^{1/\gamma_1}>\Lambda(AT)^i\right\}\right| \notag\\
		&\qquad\qquad\quad+ c(n)\left|\left\{x\in B_0:\mathcal{M}_1\left(\chi_\Omega\left|\mu\right|\right)(x)^{1/(p-1)}>\Lambda T^{-\gamma}(AT)^i\right\}\right|\notag\\
		&\qquad\qquad\quad+ \left|\left\{x\in B_0:\mathcal{M}\left(\chi_\Omega\left|\nabla u\right|^{2-p}\right)(x)^{1/(2-p)}>\Lambda T(AT)^{i+1}\right\}\right|.
	\end{align*}
Thus the proof is complete.	
\end{proof}

\section{Proof of Thereom \ref{theo_1.2}}\label{sec-4}
Let $B_0$ be a ball of radius $R_0\leq 2 \mathrm{diam}\left(\Omega\right)$ that contains $\Omega$. Then it is noticed that $\mathrm{diam}\left(\Omega\right)\leq 2R_0$. We also extend $u$ and $\mu$ to be zero in $\mathbb{R}^n\backslash\Omega$. We will show that
	\begin{align}\label{eq4.1}
		\left\|\nabla u\right\|_{L^{s,t}\left(\Omega\right)} \leq C\left\|\mathcal{M}_1\left(\left|\mu\right|\right)^{1/(p-1)}\right\|_{L^{s,t}\left(B_0\right)}
	\end{align}
where $2-p<s<p+\varepsilon$ and $0<t\leq \infty$. Here $\varepsilon>0$ is a small number depending only on $n,
p, \alpha, \beta$, and $c_0$. In what follows we consider only the case $t\neq\infty$ as for $t=\infty$ the
proof is similar. Moreover, to prove \eqref{eq4.1} we may assume that
	\begin{align*}
		\left\|\nabla u\right\|_{L^{\gamma_1}\left(\Omega\right)}\neq 0.
	\end{align*}
Let $r=\min\left\{r_0,\mathrm{diam}\left(\Omega\right)\right\}/10$. For $T>1$ we first claim that there is $\Lambda>0$ such that
	\begin{align}\label{eq4.2}
		\left|\left\{x\in\mathbb{R}^n:\mathcal{M}\left(\chi_\Omega\left|\nabla u\right|^{\gamma_1}\right)(x)^{1/\gamma_1}>\Lambda \right\}\right| < T^{-p\theta_0}r^n\left|B_1\right|.
	\end{align}
Indeed, by weak type (1,1) estimate for the maximal function we have
	\begin{align*}
	\left|\left\{x\in\mathbb{R}^n:\mathcal{M}\left(\chi_\Omega\left|\nabla u\right|^{\gamma_1}\right)(x)^{1/\gamma_1}>\Lambda \right\}\right| < \frac{c(n)}{\Lambda^{\gamma_1}}\int_\Omega\left|\nabla u\right|^{\gamma_1}dx.
	\end{align*}
By choosing $\Lambda$ such that
	\begin{align}\label{eq4.3}
		\frac{c(n)}{\Lambda^{\gamma_1}}\int_\Omega\left|\nabla u\right|^{\gamma_1}dx = T^{-p\theta_0}r^n\left|B_1\right|.
	\end{align}
Let $A$, $\theta_0>1$ be as in Lemma \ref{lem_6.4} and let $0<\gamma_1<\min\left\{\frac{p\theta_0}{p-1},\frac{(p-1)n}{n-1}\right\}$ and $\gamma=p\theta_0/\gamma_1(p-1)-1$. For $0<t<\infty$ we consider the sum
	\begin{align*}
		\mathcal{S}=\sum_{i=1}^{\infty}\left[\left(AT\right)^{si}\left|\left\{x\in B_0:\mathcal{M}\left(\left|\nabla u\right|^{\gamma_1}\right)^{1/\gamma_1}>\Lambda\left(AT\right)^i\right\}\right|\right]^{t/s}
	\end{align*}
It is noticed that 
	\begin{align}\label{eq4.4}
		C^{-1}\mathcal{S}\leq \left\|\mathcal{M}\left(\left|\nabla u/\Lambda\right|^{\gamma_1}\right)^{1/\gamma_1}\right\|_{L^{s,t}\left(B_0\right)}^t \leq C\left(\mathcal{S}+\left|B_0\right|^{t/s}\right)
	\end{align}
By \eqref{eq4.2} and Lemma \ref{lem_6.4} we have 
	\begin{align}\label{eq4.5}
		\mathcal{S}\leq& c(n)\sum_{i=1}^{\infty}\left[\left(AT\right)^{si}T^{-p\theta_0}\left|\left\{x\in B_0:\mathcal{M}\left(\left|\nabla u\right|^{\gamma_1}\right)^{1/\gamma_1}>\Lambda\left(AT\right)^{i-1}\right\}\right|\right]^{t/s}\notag\\
		&+c(n)\sum_{i=1}^{\infty}\left[\left(AT\right)^{si}\left|\left\{x\in B_0:\mathcal{M}_1\left(\left|\mu\right|\right)^{1/(p-1)}>\Lambda T^{-\gamma}\left(AT\right)^{i-1}\right\}\right|\right]^{t/s}\notag\\
		&+\sum_{i=1}^{\infty}\left[\left(AT\right)^{si}\left|\left\{x\in B_0:\mathcal{M}\left(\left|\nabla u\right|^{2-p}\right)^{1/(2-p)}>\Lambda T\left(AT\right)^{i}\right\}\right|\right]^{t/s}\notag\\
		\leq& C\left[\left(AT\right)^{s}T^{-p\theta_0}\right]^{t/s}\left(S+\left|B_0\right|^{t/s}\right) + C\left\|\mathcal{M}_1\left(\left|\mu\right|/\Lambda^{p-1}\right)^{1/(p-1)}\right\|_{L^{s,t}\left(B_0\right)}^t\notag\\
		&+\left\|\left[\mathcal{M}\left(\left|\nabla u\right|/\Lambda T\right)^{2-p}\right]^{1/(2-p)}\right\|_{L^{s,t}\left(B_0\right)}^t
	\end{align}	
Since $A$ and $C$ are fixed, then $T^{1-\frac{p\theta_0}{s}}<<1$ for $T$ large enough if $s<p\theta_0$, that is, $s<p+\varepsilon$ with $\varepsilon=p\left(\theta_0-1\right)$. In this case, the estimate \eqref{eq4.5} gives us
	\begin{align}\label{eq4.6}
	\mathcal{S}
	\leq& C\left|B_0\right|^{t/s}+ C\left\|\mathcal{M}_1\left(\left|\mu\right|/\Lambda^{p-1}\right)^{1/(p-1)}\right\|_{L^{s,t}\left(B_0\right)}^t\notag\\
	&+\left\|\left[\mathcal{M}\left(\left|\nabla u\right|/\Lambda T\right)^{2-p}\right]^{1/(2-p)}\right\|_{L^{s,t}\left(B_0\right)}^t.
	\end{align}
Combining \eqref{eq4.4} and \eqref{eq4.6} and using the boundedness of $\mathcal{M}$ on $L^{s/(2-p),t}\left(\mathbb{R}^n\right)$ where $s/(2-p)> 1$ and $t\geq 0$, we have
	\begin{align*}
		\left\|\nabla u\right\|_{L^{s,t}\left(\Omega\right)} \leq& C\left(\left|B_0\right|^{1/s}\Lambda+ \left\|\mathcal{M}_1\left(\left|\mu\right|\right)^{1/(p-1)}\right\|_{L^{s,t}\left(B_0\right)}+T^{-1}\left\|\nabla u\right\|_{L^{s,t}\left(B_0\right)}\right),
	\end{align*}
and hence for $T$ sufficiently large one has
	\begin{align}\label{eq4.7}
	\left\|\nabla u\right\|_{L^{s,t}\left(\Omega\right)} \leq& C\left(\left|B_0\right|^{1/s}\Lambda+ \left\|\mathcal{M}_1\left(\left|\mu\right|\right)^{1/(p-1)}\right\|_{L^{s,t}\left(B_0\right)}\right).
	\end{align}
We now estimate $\Lambda$. It follows from \eqref{eq4.3} and $\gamma_1<\frac{n(p-1)}{n-1}$ and the standard estimate for equations with measure data (see \cite[Theorem 4.1]{bebo}) that
	\begin{align*}
		\Lambda &\leq Cr^{-n/\gamma_1}\left\|\nabla u\right\|_{L^{\gamma_1}\left(\Omega\right)}\notag\\
		&\leq C\min\left\{r_0,\mathrm{diam}\left(\Omega\right)\right\}^{-n/\gamma_1}\mathrm{diam}\left(\Omega\right)^{n/\gamma_1}\left(\frac{\left|\mu\right|\left(\Omega\right)}{\mathrm{diam}\left(\Omega\right)^{n-1}}\right)^{1/(p-1)}.
	\end{align*}
On the other hand, since $R_0\leq 2\mathrm{diam}\left(\Omega\right)$, we have 
	\begin{align}\label{eq4.8}
		\Lambda\leq C\left(n,p,\mathrm{diam}\left(\Omega\right)/r_0\right)\mathcal{M}_1\left(\left|\mu\right|\right)(x)^{1/(p-1)}
	\end{align}
for any $x\in B_0$. Finally the proof follows from \eqref{eq4.7} and \eqref{eq4.8}.


\end{document}